\documentclass[a4paper, 12pt, oneside, reqno]{amsart}

\usepackage{
	amssymb,
	amsfonts,
	amsmath,
	amsthm,
	mathtools,
	commath,
	cases,
}

\usepackage[left = 2.5 cm, right = 2.5 cm, top = 2.5 cm, bottom = 2.5 cm]{geometry}

\usepackage[pagewise]{lineno}
\usepackage[utf8]{inputenc}	
\usepackage{hyperref, color}

\usepackage{subcaption}
\usepackage{graphicx}

\hypersetup{
	colorlinks = true,
	linkcolor = red,
	filecolor = blue,
	urlcolor = blue,
	citecolor = blue,
	pdftitle = {Domingos-Santos-Vitório}
}

\usepackage{diagbox}
\usepackage{multirow}
\usepackage{esdiff}

\newcommand{\R}{\mathbb R}

\newcommand{\Ccal}{\mathcal C}
\newcommand{\Kcal}{\mathcal K}
\newcommand{\Pcal}{\mathcal P}

\newcommand{\bs}{\bar{s}}

\DeclareMathOperator{\arcsinh}{\mathrm{arcsinh}}
\DeclareMathOperator{\intt}{\mathrm{int}}

\newtheorem{theorem}{Theorem}[section]
\newtheorem{lemma}[theorem]{Lemma}
\newtheorem{proposition}[theorem]{Proposition}
\newtheorem{corollary}[theorem]{Corollary}

\theoremstyle{definition}

\newtheorem{example}[theorem]{Example}
\newtheorem{remark}[theorem]{Remark}

\numberwithin{equation}{section}
\numberwithin{figure}{section}

\title[Rotational Ricci surfaces]{Rotational Ricci surfaces}

\author[Iury Domingos, Roney Santos and Feliciano Vit\'orio]{Iury Domingos, Roney Santos and Feliciano Vit\'orio}

\address{KU Leuven\\
	Department of Mathematics\\
	Celestijnenlaan 200B -- Box 2400, 3001  Leuven, Belgium.}
\email{iury.domingos@im.ufal.br}

\address{Universidade Federal do Cear\'a\\
	Departamento de Matem\'atica\\		
	Campus do Pici, Av. Humberto Monte, 60455-760, Fortaleza - CE, Brazil.}
\email{roneysantos@mat.ufc.br}

\address{Universidade Federal de Alagoas\\
	Instituto de Matem\'{a}tica\\
	Campus A. C. Sim\~{o}es, BR 104 - Norte, Km 97, 57072-970, Macei\'o - AL, Brazil.}	
\email{feliciano@pos.mat.ufal.br}

\thanks{I. Domingos was supported by the Research Foundation-Flanders (FWO) and the Fonds de la Recherche Scientifique (FNRS) under EOS Project G0H4518N. R. Santos was supported by Instituto Serrapilheira, grant ``New perspectives of the min-max theory for the area functional''. F. Vit\'orio was partially supported by CNPq (405468/2021- 0)}


\begin{document}
	
	\begin{abstract}
		We classify rotational surfaces in the three-dimensional Euclidean space whose Gaussian curvature $K$ satisfies
		\begin{equation*}
			K\Delta K - \|\nabla K\|^2-4K^3 = 0.
		\end{equation*}
		These surfaces are referred to as rotational Ricci surfaces. As an application, we show that there is a one-parameter family of such surfaces meeting the boundary of the unit Euclidean three-ball orthogonally. In addition, we show that this family interpolates a vertical geodesic and the critical catenoid.
	\end{abstract}
	
	\maketitle
	
	\section{Introduction}
	
	A classical problem in the theory of surfaces is to determine when a given Riemannian surface can be isometrically immersed as a minimal surface into a three-dimensional Riemannian manifold. In the Euclidean space $\R^3$, a result due to G. Ricci-Curbastro in \cite{Ricci-Curbastro} states that it is always possible provided that its Gaussian (intrinsic) curvature is strictly negative and satisfies a partial differential equation. He proved that a Riemannian surface $(\Sigma,\dif \sigma^2)$ with Gaussian curvature $K<0$ has local isometric embedding as minimal surface in $\R^3$ if and only if one of the three equivalent conditions below holds:
	\begin{itemize}
		\item[1.] the metric $\sqrt{-K}\dif \sigma^2$ is flat;
		\item[2.] the metric $(-K)\dif \sigma^2$ is round with Gaussian curvature $1$;
		\item[3.] the Gaussian curvature $K$ satisfies equation $4K = \Delta\log(-K).$
	\end{itemize}
	
	This result, however, does not include the case of zero Gaussian curvature at some point of $\Sigma.$ In \cite{MoroianuRicci}, A. Moroianu and S. Moroianu extended it by allowing this possibility. They proved that if $(\Sigma, \dif\sigma^2)$ has Gaussian curvature $K\leq 0,$ it can be isometrically locally embedded in $\R^3$ if and only if $K$ satisfies
	\begin{equation}\label{ricci-condition}
		K\Delta K - \|\nabla K\|^2-4K^3 = 0.
	\end{equation}
	
	Note that the equation above implies $4K = \Delta\log(-K)$ at the points where $K\neq 0$.
	
	A. Moroianu and S. Moroianu defined in \cite{MoroianuRicci}, inspired by the result of G. Ricci-Curbastro,  that a \emph{Ricci surface} is a Riemannian surface $(\Sigma, \dif\sigma^2)$ whose Gaussian curvature satisfies equation \eqref{ricci-condition}. In this case, $\dif \sigma^2$ is called \emph{Ricci metric} and equation \eqref{ricci-condition} is called \emph{Ricci condition}. They also proved that the Gaussian curvature of Ricci surfaces is either identically zero or has only isolated zeros and does not change sign. This property will play an important role in our work.
	
	\subsection*{Examples of Ricci surfaces}
	
	The simplest class of Ricci surfaces are the flat Riemannian surfaces.
	
	A first more interesting class of examples of Ricci surfaces are the minimal surfaces of $\R^3.$ Indeed, by the results discussed above, Ricci surfaces of $\R^3$ can be regarded as a generalisation of minimal surfaces of $\R^3$ in the following sense: every minimal surface in $\R^3$ is a Ricci surface, and every Ricci surface with non-positive Gaussian curvature admits a local isometric embedding as a minimal surface in $\R^3$.
	
	Another large class of examples of Ricci surfaces comes from biconservative surfaces in $\R^3$, as shown by D. Fetcu, S. Nistor and C. Oniciuc in \cite{MR3622313}. The authors established that Ricci surfaces arise from non-constant mean curvature biconservative surfaces of $\R^3$.
	
	Non-trivial examples of Ricci surfaces with genus at least three appear in the literature as a consequence of \cite[Theorem~1]{Traizet-2008}: for every lattice $\Gamma\subset\R^3$ and $g\geq 3$, there is a complete minimal surface $\Sigma$ in $\R^3$ invariant under the translation group defined by $\Gamma\subset\R^3$ such that $\Sigma/\Gamma$ has genus $g$ and is a closed Ricci surface, as proved by A. Moroianu and S. Moroianu in \cite{MoroianuRicci}. Additionally, in the same work, they constructed closed Ricci surfaces, possibly with conical singularities, for genus at least two whose Gaussian curvature vanishes precisely at one point.
	
	Later, in \cite{zang2022non}, Y. Zang provided examples of Ricci surfaces with catenoidal ends that are non-compact and have non-positive Gaussian curvature. More precisely, the author firstly proved a classification of Ricci surfaces with genus zero, non-positive Gaussian curvature and catenoidal ends. Secondly, he shown an existence result for Ricci surfaces that have an arbitrary number of genus and catenoidal ends.
	
	Our goal in this work is to contribute to the theory of Ricci surfaces in two ways. The first one is by providing a classification of this kind of surfaces that are rotational in $\R^3.$ From this classification, we present new examples of complete Ricci surfaces with negative Gaussian curvature that are not minimal in $\R^3$. The second one is the construction of a family of rotational Ricci surfaces intersecting the boundary of the unit Euclidean three-ball $B^3$ orthogonally.
	
	\subsection*{Organisation of the paper}
	
	We begin Section \ref{RicciRotational} by remembering definition and some properties of a rotational surface $\Sigma$ of $\R^3.$ For instance, we recall that $\Sigma$ can be entirely determined by a function $f$, and we deduce a fourth order system of ordinary differential equations that $f$ must satisfies for $\Sigma$ to be a Ricci surface. To finish this section, we will give some simple examples and properties of rotational Ricci surfaces.
	
	Thereafter, in Section \ref{reduction}, we will reduce the system previously found to a first order ordinary differential equation. Naturally, this ODE has three constant parameters, say $a,$ $b$ and $c.$ In Section \ref{dependence}, we will investigate qualitatively the dependence of $f$ of these constants. This will be important to see, in the general setting, existence and non-existence of rotational Ricci surfaces.
	
	The classification will be enunciated and proved in Section \ref{classificationresults} as described below.
	
	\begin{itemize}
		\item \textbf{Flat surfaces.} The only surfaces in this case are right circular cylinders, right circular cones and planes. Furthermore, this simplest situation is important to see that there are examples of complete, extendable and non-extendable rotational Ricci surfaces.
		
		\item \textbf{The case $a=0$.} Other than the case of flat surfaces, this is the only case with minimal surfaces, namely, catenoids.
		
		\item \textbf{The case $b=0$.} We obtain $f$ as an implicit function. For this reason, the result will be separated in two cases determined by the strictly monotonicity of $f.$ This family contains examples of complete surfaces with negative Gaussian curvature that have the shape of a funnel.
		
		\item \textbf{The case $a \neq 0$ and $b \neq  0$.} It will be necessary to reparametrise the profile curve of $\Sigma.$ The expressions of the reparametrisation and of the function $f$ will be given explicitly.
	\end{itemize}
	
	Once Ricci surfaces generalise minimal surfaces, we extend our interest in free boundary minimal surfaces in $B^3$ to Ricci surfaces inside $B^3$ meeting $\partial B^3$ orthogonally. In this way, in Section \ref{freeboundary}, as an application of our classification results, we introduce a one-parameter family of rotational Ricci surfaces with this property. The orthogonality condition imposes the strong conditions that these surfaces have negative Gaussian curvature and are such that $a=0.$ Furthermore, this family gives an interpolation between the vertical geodesic of $B^3$ and the critical catenoid.
	
	\section{Ricci condition for rotational surfaces}\label{sec:2}
	
	In this section we introduce the Ricci rotational surfaces in the Euclidean space $\R^3$. We denote by $(x,y,z)$ the canonical coordinates in $\R^3$ and, without loss of generality, we assume that the rotational axis is the $z$-axis and generating curves of our rotational surfaces are contained in the $xz$-plane.
	
	\subsection{Preliminaries and first examples}\label{RicciRotational}
	
	Let $\alpha: I \to \R^3$ be a regular curve given by
	\[\alpha(s) = (f(s), 0, g(s)),\]
	where $I\subseteq\R$ is an open interval and $f, g:I \to \R$ are smooth functions such that $f > 0.$ We assume that $\alpha$ is parametrised by its arc length, i.e.,
	\[f'(s)^2+g'(s)^2 = 1.\]
	The surface $\Sigma$ generated by the rotation of $\alpha$ around the $z$-axis, given by the parametrisation
	\[X(s, \theta)=\big(f(s)\cos \theta , f(s)\sin \theta , g(s)\big),\]
	is called a \textit{rotational surface} and $\alpha$ is its \textit{profile curve}.
	
	Since the function $g$ satisfies
	\[g(s) = g(s_0) \pm \int_{s_0}^s\sqrt{1 - f'(\zeta)^2}\dif \zeta\]
	where $s_0 \in I,$ then $\Sigma$ is determined only by $f,$ up to a vertical isometry of $\R^3$. For this reason, without loss of generality we suppose that $g(s_0)=0$, $g'\geq 0$ and write $\Sigma = \Sigma_f.$		
	
	In order to deduce the Ricci condition to these surfaces, we observe firstly that the induced metric on $\Sigma_f$ is
	\[\dif \sigma^2 = \dif s^2+f(s)^2\dif \theta^2.\]
	A direct computation shows that the Gaussian curvature $K$ of $\Sigma_f$ satisfies the ordinary differential equation
	\begin{equation}\label{eq-2}
		K(s) + \dfrac{f''(s)}{f(s)} = 0.
	\end{equation}
	Also, since the principal curvatures of $\Sigma_f$ are
	\begin{equation}\label{principal-curvatures}
		k_1(s) = -\frac{g'(s)}{f(s)} \ \text{ \ and \ } \ k_2(s)= g'(s)f''(s)-g''(s)f'(s),
	\end{equation}
	at the points where $|f'(s)|< 1$, the mean curvature of $\Sigma_f$ is given by
	\begin{equation}\label{curvatura-media}
		H(s) =\frac{f''(s)f(s)+f'(s)^2-1}{2f(s)\sqrt{1-f'(s)^2}}.
	\end{equation}
	Since $K$ does not depend of the rotation angle, we get that
	\begin{equation}\label{gradient-and-laplacian}
		\nabla K = K' X_s \ \ \text{and}\ \ \Delta K =  K''+\frac{f'}{f} K',
	\end{equation}
	and then the Ricci condition for rotational surfaces is given by
	\begin{equation}\label{eq-1}
		K K''-(K')^2-4K^3+\frac{f'}{f} K K' = 0.
	\end{equation}
	Hence, $\dif\sigma^2$ is a Ricci metric if it obeys the following system of ordinary differential equations
	\begin{equation}\label{*}\tag{$*$}
		\begin{cases}
			K K''-(K')^2-4K^3+\dfrac{f'}{f} K K' = 0,\\
			\hfill K + \dfrac{f''}{f} = 0.
		\end{cases}
	\end{equation}
	
	A rotational surface $\Sigma_f$ whose the induced metric satisfies the system $\eqref{*}$ is said a \textit{rotational Ricci surface}.
	
	We present below the first examples of rotational Ricci surfaces. As we will see, these are the only examples with vanishing Gaussian curvature in at least one point of the surface.
	
	\begin{example}[Complete case]
		For each $r > 0,$ the simplest examples of complete rotational Ricci surfaces are the right circular cylinder of radius $r,$ and the catenoid with neck of radius $r$, centered at the origin of $\R^3$. We will denote these surfaces by $\Ccal_0(r)$ and $\Ccal_1(r)$, respectively.
	\end{example}
	
	\begin{example}[Extended case]
		For each $z_0\in \R$, the plane $\{z=z_0\}$ is of course a Ricci surface, but it is not a rotational surface according to our definition, as in this case, $f$ would vanish at a point. However, it is an \textit{extended rotational surface} in the sense of \cite[Remark~4, p. 77]{manfredo}. Therefore, we consider these horizontal planes as rotational Ricci surfaces, and we will denote them by $\Pcal_{z_0}.$ %Therefore, $\{z=z_0\}$ is a complete rotational Ricci surface, denoted by $\mathcal{P}_{z_0}$.
	\end{example}
	
	\begin{example}[Non-extendable case]
		For each $\varphi\in(0,\pi)$ and $v\in \R$, the simplest examples of non-extendable rotational Ricci surfaces are the cones of opening angle $\varphi$ and vertex at the $(0,0,v).$ In this case, $f:I \to \R$ is given by $f(s) = ms+ r$ and $g(s)=\sqrt{1-m^2}(s-s_0)$ for some $s_0 \in I,$ with $0<m^2 <1$. Then $\varphi=2\arctan(\tfrac{m}{\sqrt{1-m^2}})$ and $v = -\tfrac{\sqrt{1 - m^2}}{m}(ms_0 + r).$ We will denote these cones by $\Kcal_v(\varphi).$	
	\end{example}
	
	By the result of A. Moroianu and S. Moroianu, we know that the Gaussian curvature of a rotational Ricci surface either vanishes identically or has only isolated zeros and does not change sign on $\Sigma_f$ \cite[Theorem~1.2]{MoroianuRicci}. We can use this to prove the following properties about the curvature of $\Sigma_f.$
	
	\begin{proposition}\label{vetor-horizontal}
		Let $\Sigma_f$ be a rotational Ricci surface, where $f:I \to \R$ is a positive smooth function. Let $s_*\in \bar{I}$ be such that $f$ is well-defined and positive at $s_*.$
		\begin{itemize}
			\item[1.] If $K(s_*) = 0,$ then $K$ vanishes identically.
			\item[2.] If $\alpha'(s_*)$ is a horizontal vector, then $\Sigma_f$ is part of a horizontal plane.
		\end{itemize}
	\end{proposition}
	
	\begin{proof}
		\noindent 1. We have that $K$ vanishes along the curve $\{X(s_*, \theta) : \theta \in \R\}$ since $\Sigma_f$ is rotational, and it is a subset of non-isolated points of $\Sigma_f.$ Therefore, since $\Sigma_f$ is a Ricci surface, $K$ vanishes on $\Sigma_f.$\\
		
		\noindent 2. Since $f'(s_*)=\pm 1$ and $g'(s_*)=0,$ then by \eqref{principal-curvatures} we get that $K(s_*)=0$. It follows that $K$ is identically zero on $\Sigma_f$. That is, $f''=0$ and therefore $f(s) = ms + r$, with $m^2=1$, and $g(s)=0$, up to a vertical isometry of $\R^3$.
	\end{proof}
	
	\subsection{Reduction to a first order ODE}\label{reduction}
	
	In this subsection, we reduce the system \eqref{*} to a first order ordinary differential equation on $f$. Due to the nature of this system, we obtain a differential equation depending on three parameters $(a,b,c)$. This is the content of the next result.
	
	\begin{lemma}\label{lemma}
		
		Let $\Sigma_f$ be a rotational Ricci surface in $\R^3,$ where $f: I \to \R$ is a positive smooth function. Then, there exist constants $a,b,c\in \R$ such that
		\begin{equation}\label{ode-ricci}
			f(s)f'(s) = af(s)+bs+c.
		\end{equation}
		
		Conversely, let $f: I\to\R$ be a positive smooth function defined in an interval $I\subseteq \R$ such that
		$|f'(s)|\leq 1$ and \eqref{ode-ricci} holds for some $a,b,c\in\R$. Then there is a unique rotational Ricci surface $\Sigma_f$, up to a vertical isometry of $\R^3$, whose profile curve $\alpha:I\to\R^3$ is given by
		\begin{equation}\label{profile-curve}
			\alpha(s)=\left(f(s),0,\int_{s_0}^s\sqrt{\frac{f(\zeta)^2 - \left(af(\zeta) + b\zeta + c\right)^2}{f(\zeta)^2}}\dif \zeta\right),
		\end{equation}
		for some $s_0 \in I.$
	\end{lemma}
	
	\begin{proof}
		Since $\Sigma_f$ is a rotational Ricci surface, then either $K$ vanishes identically, or it has only isolated zeros and does not change sign. If $K$ vanishes identically on $\Sigma_f$, then $f$ must be an affine function and thus equation \eqref{ode-ricci} holds for some constants $a,b,c\in \R$.
		
		Suppose that $K$ does not vanish identically on $\Sigma_f$. Since $K$ has isolated zeros on $\Sigma_f$, we restrict ourselves on an open set $U\subseteq\Sigma$ with $K \neq 0$.
		Multiplying equation \eqref{eq-1} by $f/K^2$, on $U$ we get
		\[\left(f \frac{K'}{K}\right)' -4 fK =0.\]
		Using equation \eqref{eq-2} in the term $-4 fK$, a first integration implies that
		\[f \frac{K'}{K} + 4 f'-a=0 \ \text{for some} \ a\in \R.\]
		Once $K$ does not change of sign on $\Sigma_f$, dividing the equation above by $f$, we get on $U$ that
		\[\left(\log |K|\right)' + 4 (\log f)'-\frac{a}{f}=0,\]
		that is,
		\[\left(\log (|K| f^4)\right)'-\frac{a}{f}=0.\]
		Therefore, by equation \eqref{eq-2}, we get
		\begin{equation}\label{prooflemma}
			ff'''+3f'f''-af''=0
		\end{equation}
		on $U$. As $ff'''+3f'f'' = (ff'')'+\big((f')^2\big)'$, a second integration implies that
		\[ff''+(f')^2-af'-b=0 \ \text{for some} \ b\in \R,\]
		and since $ff''+(f')^2 = (ff')'$, a third integration implies the first part of our assertion.		
		
		We now prove the second part of the Lemma. Let $f: I\to\R$ be a positive smooth function, such that
		$|f'(s)|\leq 1$, satisfying equation \eqref{ode-ricci}, for some $a,b,c\in\R$. Consider $\Sigma_f$ the rotational surface whose profile curve $\alpha$ is given by \eqref{profile-curve}. Note that $\alpha$ is well-defined since $|f'(s)|\leq 1$. By equation \eqref{ode-ricci}, we have that
		\begin{equation*}
			K(s) = \frac{(bs+c)^2+af(s)(bs+c)-bf(s)^2}{f(s)^4},
		\end{equation*}
		and a direct long computation with \eqref{ode-ricci} and the expression above implies that
		\begin{align*}
			K'(s) &=\frac{3af(s)+4(bs+c)}{f(s)^2}K(s),\\
			K''(s) &=\frac{24(bs+c)^2+35a(bs+c)f(s)+4(3a^2-b)f(s)^2}{f(s)^4}K(s)
		\end{align*}		
		By equation \eqref{eq-1}, we conclude that $K$ satisfies the Ricci condition. Therefore, $\Sigma_f$ is a rotational Ricci surface. Moreover, $\Sigma_f$ is uniquely determined by $f$, up to a vertical isometry of $\R^3$.
	\end{proof}
	
	\begin{remark}
		The statement of Lemma \ref{lemma} is intrinsic in the following sense. Let $(\Sigma,\dif \sigma^2)$ be a rotationally symmetric Riemannian surface. In terms of geodesic polar coordinates $(s,\theta)$, we can write its metric locally as
		\[\dif\sigma^2 = \dif s^2 + f(s)^2\dif\theta^2,\]
		for some positive smooth function $f:I \to \R$. Then $\dif\sigma^2$ is a Ricci metric if and only if the function $f$ satisfies equation \eqref{ode-ricci}.
	\end{remark}
	
	\begin{remark}\label{remark-derivadas}
		We also observe that the condition $|f'(s)| \leq 1$ is equivalent to
		\begin{equation}\label{condition-f'}
			\left(af(s)+bs+c\right)^2 \leq f(s)^2.
		\end{equation}
		By Proposition \ref{vetor-horizontal}, the equality holds if and only if $\Sigma_f$ is part of a horizontal plane. This condition also implies that the maximal set of definition of the profile curve $\alpha$ of $\Sigma_f$ is $\{s\in \R : f(s)  \ \text{is well-defined}\}\cap\{s\in \R : \eqref{condition-f'} \ \text{holds}\}$.
	\end{remark}
	
	\begin{remark}
		By Lemma \ref{lemma}, the Gaussian curvature of $\Sigma_f$ is given by
		\begin{equation}\label{eq-K_f}
			K(s) = \frac{(bs+c)^2+af(s)(bs+c)-bf(s)^2}{f(s)^4}.
		\end{equation}
		At the points where  $|f'(s)|<1$, we compute the mean curvature of $\Sigma_f$ using equation \eqref{curvatura-media} as
		\begin{align}
			\label{eq-H_f} H(s) &= \frac{\left(b-1\right) f(s)+a (af(s)+b s+c)}{2 f(s) \sqrt{f(s)^2-(a f(s)+b s+c)^2}}.
		\end{align}
		Moreover, the tangent vector of the profile curve of $\Sigma_f$ is horizontal when $f'(s)\to\pm 1$. In these points, the mean curvature of $\Sigma_f$ goes to infinity.	
	\end{remark}	
	
	\subsection{The $(a,b,c)$-dependence of $\Sigma_f$}\label{dependence}
	
	In order to guarantee the existence and uniqueness of $f$ for a such choice of parameters $(a,b,c)$, in this subsection we study the range of possibility for them excluding some sets of $\R^3$.
	
	Let $\mathcal{E}=\cup \mathcal{E}_i$ where $\mathcal{E}_1$, $\mathcal{E}_2$ and $\mathcal{E}_3$ are the following disjoint subsets of $\R^3$:
	\begin{align*}
		\mathcal{E}_1 &=\{(x_1,0,0)\in\R^3 : |x_1|\geq 1\},\\
		\mathcal{E}_2 &=\{(x_1,0,x_3)\in\R^3 : x_1\geq  1 \text{ and } x_3>0\}, \\
		\mathcal{E}_3 &=\{(x_1,0,x_3)\in\R^3 : x_1\leq -1 \text{ and } x_3<0\}.
	\end{align*}
	For a given $(a,b,c)\in \R^3$, we consider the open set $\Omega\subset\R\times(0,+\infty)$ defined by
	\[\Omega = \{(s,x)\in\R\times(0,+\infty) : (ax+bs+c)^2 < x^2\}.\]
	A straightforward verification infers that $\Omega \neq \emptyset$ if and only if $(a,b,c)\notin \mathcal{E}$.
	
	We remark some geometric aspects of the $\Omega$ set. Firstly, we suppose that $a^2\neq 1$. We consider $\ell_1$ and $\ell_2$ straight lines on the $sx$-plane given by
	\[\begin{array}{cl}
		\ell_1: & (a-1)x+bs+c=0,\\
		\ell_2: & (a+1)x+bs+c=0.\\
	\end{array}\]
	When $b\neq 0$, they are intersecting lines with intersection point at $s$-axis. If $\ell_1$ and $\ell_2$ have slopes with different signs (i.e., $a^2<1$), then the $s$-coordinate on $\Omega$ is allowed to assume any real value. Otherwise, if $\ell_1$ and $\ell_2$ have slopes with same sign (i.e., $a^2>1$), there is a natural barrier $\omega_s=-\frac{c}{b}$ for the $s$-coordinate on $\Omega$. When $b=0$, we have that $\ell_1$, $\ell_2$ and $s$-axis are parallels, then the $s$-coordinate on $\Omega$ is allowed to assume any real value, and the $x$-coordinate on $\Omega$ is ever away from zero when $c\neq 0$. In particular, if $c=0$ then $a^2<1$ and $\Omega = \R\times(0,+\infty)$.
	
	Suppose that $a^2=1$. When $b=0$, then the $s$-coordinate on $\Omega$ is allowed to assume any real value; %In particular, for $c=0$ we have $\Omega = \R\times (0,+\infty)$.
	moreover, since $c\neq 0$ in this case, there is a natural barrier $\omega_x=-\frac{c}{2a}$ for the $x$-coordinate on $\Omega$ and then $\Omega = \R\times(\omega_x,+\infty)$. When $b\neq 0$, we have the straight lines $\ell_1$ and $\ell_2$ on the $sx$-plane
	\[\begin{array}{cl}
		\ell_1: & bs+c=0,\\
		\ell_2: & \pm 2x+bs+c=0.\\
	\end{array}\]
	They are intersecting lines with intersection point at $s$-axis, and there is a natural barrier $\omega_s=-\frac{c}{b}$ for the $s$-coordinate on $\Omega$, and $\Omega$ is the region on the $sx$-plane between the straight lines $\ell_1$ and $\ell_2$ such that $x>0$.
	
	Therefore, given $(a,b,c)\in\R^3\setminus \mathcal{E}$ and $(s_0,x_0)\in\Omega$, we define the smooth function $F:\Omega \to \R$ given by
	\[F(s,x) = \frac{ax + bs + c}{x}.\]
	We consider the following initial value problem
	\begin{equation*}
		\begin{cases}\label{PVI}\tag{$**$}
			x'(s) = F(s,x(s)),\\
			x(s_0) = x_0.
		\end{cases}
	\end{equation*}
	By Picard-Lindel\"of theorem, there exists a unique maximal smooth solution $f: I\to\R$ of \eqref{PVI} such that $f$ is a positive function, $|f'(s)|<1$ and equation \eqref{ode-ricci} holds. Moreover, since $f$ uniquely determined by \eqref{PVI} then $\Sigma_f$ is uniquely determined by $(a,b,c)$ and $(s_0,x_0)$, up to a vertical isometry of $\R^3$.
	
	\begin{remark} Since $\dif\sigma^2$ is complete if and only if $I = \R,$ we observe that if $a^2\geq 1$ and $b\neq 0$ then the Ricci metric $\dif\sigma^2$ cannot be complete, once there are natural barriers $\omega_s$'s on $\Omega$ in these cases. As we shall subsequently see,  in our classification results does not appear complete rotational Ricci surfaces of positive Gaussian curvature, and even in the negative Gaussian curvature case, only some examples are complete.
	\end{remark}
	
	\section{Classification of rotational Ricci surfaces}\label{classificationresults}
	
	In this section, we classify the rotational Ricci surfaces in $\R^3$ studying the solutions of equation \eqref{ode-ricci} and giving some geometric properties of them.  For the sake of completeness, let us treat the simple case when the Gaussian curvature $K$ of $\Sigma_f$ vanishes identically.
	
	\begin{proposition}\label{simplest}
		Let $\Sigma_f$ be a rotational surface in $\R^3,$ where $f:I \to \R$ is a positive smooth function. Suppose that the Gaussian curvature of $\Sigma_f$ vanishes identically. Then $f$ is an affine function and $\Sigma_f$ is part of either a right circular cylinder $\mathcal{C}_0(r)$, or a cone $\Kcal_v(\varphi)$ or an extended plane $\Pcal_0$.	
	\end{proposition}
	
	\begin{proof}
		Since $K$ vanishes on $\Sigma_f$ then $f''=0$ and therefore $f(s) = ms + r$, with $m^2\leq 1$, and $g(s)=\sqrt{1-m^2}(s-s_0).$ If $m=0$, then $r>0$ and $\Sigma_f$ is part of a right circular cylinder $\mathcal{C}_0(r)$. If $m^2=1$, then $g$ vanishes identically on $I$ and $\Sigma_f$ is part of an extended plane $\Pcal_0$. If $0<m^2<1$, then  $\Sigma_f$ is part of a cone $\Kcal_v(\varphi)$ with $\tan\tfrac{\varphi}{2}=\tfrac{m}{\sqrt{1-m^2}}$ and $v = -\tfrac{\sqrt{1 - m^2}}{m}(ms_0 + r),$ for some $s_0\in I.$
	\end{proof}
	
	From now on, we suppose that $\Sigma_f$ is a rotational Ricci surface in $\R^3$ whose Gaussian curvature does not vanish identically. In order to classify these surfaces, we use Lemma \ref{lemma} to distinguish the three cases.
	
	\subsection{Case \texorpdfstring{$a = 0$}{a=0}}
	
	The rotational Ricci surfaces $\Sigma_f$ for which $a = 0$ in Lemma \ref{lemma} are characterised by  $f(s)f'(s) = bs + c$. By integration,
	\[f(s)^2 = bs^2 + 2cs + d\]
	for some $d\in\R$ such that $bs^2 + 2cs + d > 0$ for each $s\in I$. Note that we do not lost generality if we assume that $b \neq 0$ because this case was treated in the Proposition \ref{simplest}. By equation \eqref{eq-K_f}, the Gaussian curvature of $\Sigma_f$ is
	\begin{equation*}
		K(s) = \frac{c^2 - bd}{\left(bs^2 + 2cs + d\right)^2}.
	\end{equation*}
	Note that $K$ depends only on the sign of the constant term $c^2 - bd.$ We classify these surfaces in terms of the strict sign of $K$.
	
	\begin{theorem}[Case $a = 0$]\label{a=0}
		Let $\Sigma_f$ be a rotational Ricci surface in $\R^3,$ where $f:I \to \R$ is a positive smooth function defined by
		\[f(s) = \sqrt{bs^2 + 2cs + d}\]
		for some $b, c,d \in \R$ with $b \neq 0.$ Suppose that the Gaussian curvature $K$ of $\Sigma_f$ does not vanish and set
		\[s_1 = -\frac{1}{b}\left(c + \sqrt{\frac{bd - c^2}{b - 1}}\right)\ \ \mbox{and}\ \ \ s_2 = -\frac{1}{b}\left(c - \sqrt{\frac{bd - c^2}{b - 1}}\right)\]
		when $(bd - c^2)(b - 1) > 0.$ Then, one of the following holds.
		\begin{itemize}
			\item[1.] If $K<0,$ then $b, d \in (0, +\infty)$, $c^2<bd$ and			
			\begin{itemize}
				\item[$-$] either $I\subseteq\R$ for $0<b\leq 1$;
				
				\item[$-$] or $I\subseteq(s_1, s_2)$ for $b > 1.$
			\end{itemize}
			
			\item[2.] If $K>0,$ then $b \in (-\infty, 1)$, $bd<c^2$ and
			\begin{itemize}
				\item[$-$] either $I \subseteq (s_1, s_2)$ for $b < 0$;
				
				\item[$-$] or $I \subseteq (\tfrac{c^2-d}{2c},+\infty)$ if $c>0,$ or $I \subseteq (-\infty,\tfrac{c^2-d}{2c})$ if $c<0$, for $b = 0$;
				
				\item[$-$] or $I \subseteq (-\infty, s_1)$, or $I \subseteq (s_2, +\infty)$ for $0 < b < 1.$
			\end{itemize}
		\end{itemize}
	\end{theorem}
	
	\begin{proof}
		
		Let us consider the polynomials $P$ and $Q$ given by
		\[P(s) = b(1 - b)s^2 + 2(1 - b)cs - c^2 + d\]
		and
		\[Q(s) = bs^2 + 2cs + d.\]
		If we denote by $\Delta_P$ and $\Delta_Q$ the discriminants of $P$ and $Q$, respectively, one can check that $\Delta_Q = 4(c^2-bd)$ and $\Delta_P = (1 - b)\Delta_Q.$ Furthermore, by Remark \ref{remark-derivadas}, since $f(s) = \sqrt{Q(s)},$ the interval $I$ of definition of $\alpha$ is determined by the inequalities $P(s) \geq 0$ and $Q(s) > 0.$ But, a direct computations shows that $P(s) = Q(s) - (bs + c)^2 \leq Q(s),$ which gives
		\begin{equation}\label{maximalset}
			I \subseteq \{s\in \R : P(s) \geq 0\}.
		\end{equation}
		
		\noindent 1. Assuming that $K<0$, we have $0\leq c^2<bd.$ This implies that $b$ and $d$ must have the same sign and are non-zero. In this case, $\Delta_Q < 0,$ which means that $Q$ has no roots. Thus, by \eqref{maximalset}, the maximal set of definition of $\alpha$ is the subset of $\R$ such that $P$ is non-negative.
		\begin{itemize}
			\item[$-$] Assume $0< b \leq 1.$ Then $\Delta_P \leq 0$ and consequently the maximal set of definition of $\alpha$ is $\R$.
			
			\item[$-$] Assume $b>1.$ This means that $\Delta_P > 0$, and thus, since $b>0,$ we must have $I \subseteq (s_1, s_2)$ where $s_1$ and $s_2$ are the roots of $P$. It is direct to see that
			\[s_1 = -\frac{1}{b}\left(c + \sqrt{\frac{bd - c^2}{b - 1}}\right)\ \ \mbox{and}\ \ \ s_2 = -\frac{1}{b}\left(c - \sqrt{\frac{bd - c^2}{b - 1}}\right).\]
			
		\end{itemize}
		
		\noindent 2. Assuming that $K>0$, we have $c^2 > bd$, that is, $\Delta_Q > 0.$ If $b > 1,$ then $b(1-b)<0$ and $\Delta_P < 0$, which means that $P$ is negative everywhere. Moreover, if $b=1$ then $P$ is non-negative if and only if $d\geq c^2$ and it cannot occur by the assumption on the Gaussian curvature of $\Sigma_f.$ These contradictions show that $b<1$.	Let $s_1$ and $s_2$ be the roots of $P$ given as before. Thus, we have the following cases.
		
		\begin{itemize}
			\item[$-$] If $0<b<1,$ we get either $I \subseteq (-\infty, s_1)$ or $I \subseteq(s_2, +\infty)$.
			
			\item[$-$] If $b < 0,$ we get $I \subseteq (s_1, s_2)$.
			
			\item[$-$] If $b=0$, we get $c\neq 0$ and $P(s) = 2cs - c^2 + d.$ Then $c>0$ implies $I \subseteq (\tfrac{c^2-d}{2c},+\infty)$, while $c<0$ implies $I \subseteq (-\infty,\tfrac{c^2-d}{2c})$.
		\end{itemize}
	\end{proof}
	
	The complete surfaces from the previous result have some geometric properties that will be useful in our application.
	
	\begin{theorem}\label{a=0prop}
		Let $\Sigma_f$ be a rotational Ricci surface in $\R^3$ with negative Gaussian curvature, where $f:\R \to \R$ is a positive smooth function given by
		\[f(s) = \sqrt{bs^2 + 2cs + d}\]
		for $0 < b \leq 1$ and $c^2 < bd.$ Then:
		\begin{itemize}
			\item[1.] $\Sigma_f$ is a proper surface;
			\item[2.] $\Sigma_f$ is a symmetric surface with respect to $xy$-plane;
			\item[3.] $\Sigma_f$ is a catenoid if and only if $b = 1.$
		\end{itemize}
	\end{theorem}
	
	\begin{proof}
		Since $b > 0,$ we can consider the change of parameter $s(t) = t-\frac{c}{b}$. In this case, one can check that for $\hat{f} = f\circ s$ and $\hat{g} = g\circ s$ we have
		\begin{align}
			\hat{f}(t)&=\sqrt{\frac{b^2t^2+bd-c^2}{b}},\label{fparametro-t}\\
			\hat{g}(t)&=\int_{t_0}^t\sqrt{\frac{b^2(1-b)\tau^2
					+bd-c^2}{b^2\tau^2+bd-c^2}}\dif\tau\label{gparametro-t},
		\end{align}
		where $t_0 = s_0+\frac{c}{b}.$ Hence, by \eqref{fparametro-t}, $\hat{f}$ is an even function well-defined on $\R$ and goes to infinity as $t \to \pm\infty.$ If $0< b \leq 1,$ then $\hat{g}$ is also well-defined on $\R$ by \eqref{gparametro-t}. We set $t_0 = 0$ and observe that for $t\geq 0,$
		\begin{align}\label{g-function}
			\nonumber	\hat{g}(t) &= \int_{0}^t\sqrt{\frac{b^2(1-b)\tau^2
					+bd-c^2}{b^2\tau^2+bd-c^2}}\dif\tau\\
			&\geq \int_{0}^t\sqrt{\frac{bd-c^2}{b^2\tau^2+bd-c^2}}\dif\tau\\
			\nonumber	&= \frac{\sqrt{bd-c^2}}{b}\arcsinh\left(\frac{bt}{\sqrt{bd-c^2}}\right).
		\end{align}
		Moreover, $\hat{g}$ is an odd function since $\hat{g}$ and $\widetilde{g}(t):=-\hat{g}(-t)$ are both solutions of the ODE
		\begin{equation*}
			\begin{aligned}
				\hat{g}'(t) &= \sqrt{\frac{b^2(1-b)t^2 +bd-c^2}{b^2t^2+bd-c^2}},\\
				\hat{g}(0) &= 0.
			\end{aligned}
		\end{equation*}
		Therefore, $\hat{g}(t)\to\pm \infty$ when $t\to\pm \infty$, respectively. This proves that $\Sigma_f$ is proper and symmetric with respect to $xy$-plane.
		
		By \eqref{eq-H_f}, we have that the mean curvature of $\Sigma_f$ is
		\[H(t) = \frac{(b - 1)\sqrt{b}}{2\sqrt{b^2(1-b)t^2+bd-c^2}}.\]
		In particular, $\Sigma_f$ is minimal if and only if $b=1$, since $b>0$.
	\end{proof}
	
	\begin{figure}[!h]
		\centering
		{\includegraphics[scale=.75]{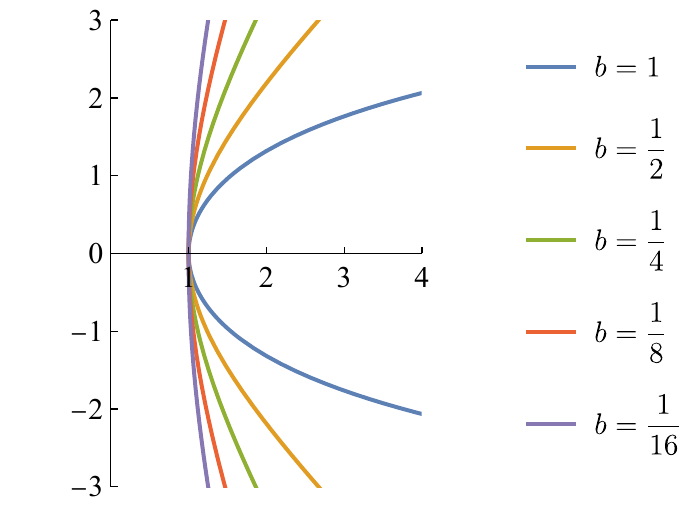}}
		\hfil
		{\includegraphics[scale=.75]{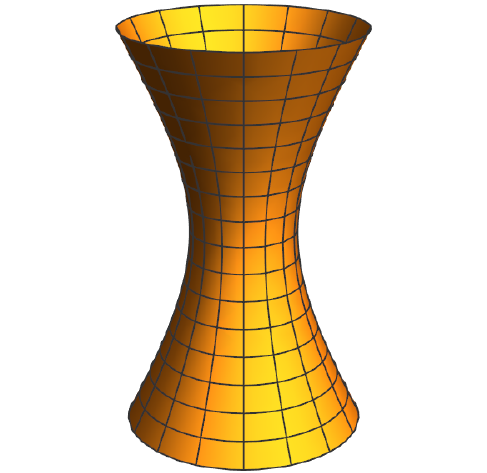}}
		\caption{Catenoidal-Ricci surface and profile curves for $(c,d) = (0,1).$}
	\end{figure}
	
	\begin{figure}[!h]
		\centering
		{\includegraphics[scale=.75]{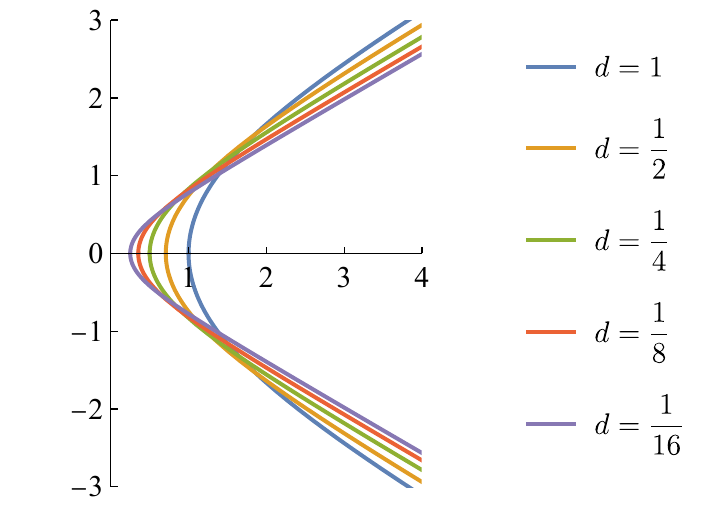}}
		\hfil
		{\includegraphics[scale=.75]{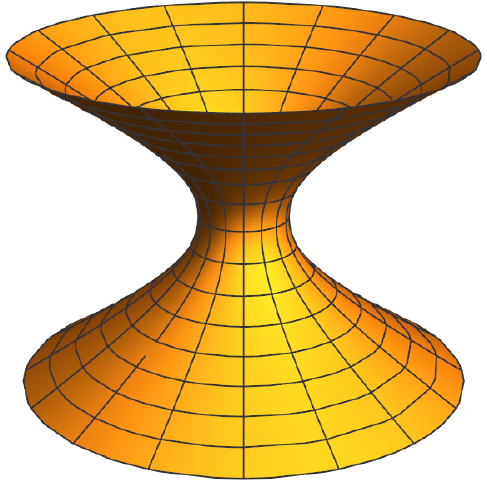}}
		\caption{Catenoidal-Ricci surface and profile curves for $(b,c)=(\frac{3}{4},0).$}
	\end{figure}
	
	The surfaces of Theorem \ref{a=0prop} will be referred as \textit{catenoidal-Ricci surfaces} and denoted by $\Ccal_b(r),$ where $r$ is \textit{neck radius} of the surface given by the positive solution of $r^2 = \frac{bd - c^2}{b}.$ Geometrically, the constant $r$ is the radius of the circle of $\Ccal_b(r) \cap \{z = 0\}.$
	
	For $c=0,$ taking the limit $b \to 0$ in Theorem \ref{a=0prop}, we have that $f(s) \to \sqrt{d},$ and therefore we get a right circular cylinder by Proposition \ref{simplest}. As a consequence, we obtain an interpolation between a cylinder $\Ccal_0(\sqrt{d})$ (which is a constant mean curvature surface) and a catenoid $\Ccal_1(\sqrt{d})$ (which is a minimal surface) by catenoidal-Ricci surfaces.
	
	\begin{corollary}
		The one-parameter family $\{\Ccal_b(\sqrt{d})\}_{0 < b \leq 1}$ of catenoidal-Ricci surfaces interpolates the cylinder $\Ccal_0(\sqrt{d})$ and the catenoid $\Ccal_1(\sqrt{d}).$
	\end{corollary}
	
	\subsection{Case \texorpdfstring{$b = 0$}{b=0}}
	
	The surfaces $\Sigma_f$ for which $a \neq 0$ and $b = 0$ in Lemma \ref{lemma} are characterised by  $f(s)f'(s) = af(s) + c$. Firstly
	note that $f(s) = -\tfrac{c}{a}$ is a trivial solution in this case. However, this solution produces surfaces whose Gaussian curvature is identically zero, which is the situation studied in Proposition \ref{simplest}. This allows us to suppose that $af(s)+c$ is not identically zero on $I$. We claim that in this case $af(s)+c \neq 0$ for each $s \in I$. In order to show this, consider the open subset $J\subseteq I$ defined by
	\[J = \{s \in I : af(s)+c\neq 0\}.\]
	By integration, one can check that for each $s \in J,$
	\begin{equation}\label{solution}
		af(s) - c\log|af(s) + c| = a^2s + d,
	\end{equation}
	for some $d\in \R$. The Gaussian curvature of $\Sigma_f$ for each $s \in J$ is given by
	\[K(s) = \frac{c\left(af(s)+c\right)}{f(s)^4}.\]
	Therefore, we also have that $c \neq 0.$ If $af(s_*) + c = 0$ for some $s_* \in I,$ then there exists a sequence $\{s_n\}_n\subseteq J$ that converges to $s_*$ such that, by \eqref{solution}, we have
	\begin{equation}\label{sstar}
		s_* = \lim_{n \to +\infty}\frac{1}{a^2}\left(af(s_n) - c\log|af(s_n) + c| - d\right) = \pm \infty
	\end{equation}
	depending on the sign of $c$. This is a contradiction and it shows our claim. It follows that the expression $af(s) + c$ has the same sign at each $s \in I.$
	
	\begin{theorem}[Case $b = 0$]
		Let $\Sigma_f$ be a rotational Ricci surface in $\R^3,$ where $f:I \to \R$ is a positive smooth function defined implicitly by
		\[af(s) - c\log\left(af(s) + c\right) = a^2s + d\]
		for some $a, c, d \in \R$ with $a \neq 0 \neq c$ such that $af(s) + c > 0$ for each $s \in I.$ Suppose that the Gaussian curvature $K$ of $\Sigma_f$ does not vanish and set
		\[s_0 = \frac{c}{a(1 - a)} - \frac{c}{a^2}\log\left(\frac{c}{1 - a}\right) - \frac{d}{a^2}\]
		when $(1-a)c>0.$ Then, $f$ is a strictly increasing diffeomorphism on $I,$ and one of the following holds.
		\begin{itemize}
			\item[1.] If $K < 0,$ then $a \in (0, +\infty),$ $c \in (-\infty, 0),$ $f$ is asymptotic to $-\frac{c}{a}$ and
			\begin{itemize}
				\item[$-$] either $I \subseteq \R$ and $f(I) \subseteq (-\frac{c}{a}, +\infty)$ for $0 < a \leq 1$;
				
				\item[$-$] or $I \subseteq (-\infty, s_0)$ and $f(I) \subseteq (-\tfrac{c}{a},\frac{c}{1-a})$ for $a > 1$.
			\end{itemize}
			
			\item[2.] If $K > 0,$ then $a \in (-\infty, 0) \cup (0, 1),$ $c \in (0, +\infty)$ and
			\begin{itemize}
				\item[$-$] either $I \subseteq (s_0, +\infty)$ and $f(I) \subseteq (\frac{c}{1-a},-\tfrac{c}{a})$ for $a < 0$;
				
				\item[$-$] or $I \subseteq (s_0, +\infty)$ and $f(I) \subseteq (\frac{c}{1-a}, +\infty)$ for $0 < a < 1$.
			\end{itemize}
		\end{itemize}
	\end{theorem}
	
	\begin{proof}
		Since $f>0$ on $I,$ then $f'>0$ on $I.$ This guarantees that $f$ is strictly increasing on $I,$ the inverse of $f$ exists and is also smooth by \eqref{solution}. Moreover, by Remark \ref{remark-derivadas} the profile curve of $\Sigma_f$ is defined in the interval $I$ such that
		\begin{equation}\label{inequalities}
			I \subseteq \{s\in \R: 0<af(s)+c < f(s)\}.
		\end{equation}
		
		\noindent 1. Assuming that $K < 0,$ we have $c < 0.$ In this case, \eqref{inequalities} implies that $a > 0$ and $-\frac{c}{a} < f(s)$. Furthermore, by a similar argument used in \eqref{sstar}, we have $s\to -\infty$ when $af(s) + c\to 0$.
		
		\begin{itemize}
			\item[$-$] If $0 < a \leq 1,$ then \eqref{solution} is well-defined for each real number because $f$ is a strictly monotone function. If we admit that $af(s) + c \to 0$ when $s \to +\infty,$ equation \eqref{solution} implies
			\[- c = \lim_{s \to +\infty}af(s) = \lim_{s \to +\infty} \left(c\log\left(af(s) + c\right) + a^2s + d\right) = +\infty,\]
			which is an absurd. Therefore, $af(s) + c$ does not approach zero when $s \to +\infty$ and consequently, by the limit before, we have that $af(s) \to +\infty.$ Since $a$ is positive, then $f(s) \to +\infty.$ We conclude that $f$ is a diffeomorphism at most from $\R$ to $(-\tfrac{c}{a},+\infty).$
			
			\item[$-$] If $a > 1,$ then by \eqref{inequalities} we get $f(s) < \frac{c}{1-a}$. Since $f(s) \to \frac{c}{1-a}$ when $s \to s_0,$ we conclude that $f$ is a diffeomorphism at most from $(-\infty,s_0)$ to $(-\tfrac{c}{a},\frac{c}{1-a}).$
		\end{itemize}
		
		\noindent 2. Assuming that $K > 0,$ we have $c > 0.$ In this case, \eqref{inequalities} implies that $a < 1$ and $\frac{c}{1-a} < f(s)$. Furthermore, $f(s) \to \frac{c}{1-a}$ when $s \to s_0.$
		
		\begin{itemize}
			\item[$-$] Assume that $a < 0.$ As argued in the calculation of \eqref{sstar}, we have that $s\to +\infty$ when $af(s) + c\to 0$. Therefore, $f$ is a diffeomorphism at most from $(s_0,+\infty)$ to $(\frac{c}{1-a},-\tfrac{c}{a}).$
			
			\item[$-$] If $0 < a < 1,$ by using the same argument as before, one can see that $f(s) \to +\infty$ when $s\to+\infty,$ and this means that $f$ is a diffeomorphism at most from $(s_0, +\infty)$ to $(\tfrac{c}{1-a},+\infty).$
		\end{itemize}
	\end{proof}
	
	By making evident modifications in the proof, we can state the following result considering that $af(s) + c$ is negative on $I.$
	
	\begin{theorem}[Case $b = 0$]
		Let $\Sigma_f$ be a rotational Ricci surface in $\R^3,$ where $f:I \to \R$ is a positive smooth function defined implicitly by
		\[af(s) - c\log\left(- af(s) - c\right) = a^2s + d\]
		for some $a, c, d \in \R$ with $a \neq 0 \neq c$ such that $af(s) + c < 0$ for each $s \in I.$ Suppose that the Gaussian curvature $K$ of $\Sigma_f$ does not vanish and set
		\[s_0 = \frac{c}{a(1 + a)} - \frac{c}{a^2}\log\left(- \frac{c}{1 + a}\right) - \frac{d}{a^2}\]
		when this number is well-defined. Then, $f$ is a strictly decreasing diffeomorphism on $I,$ and one of the following holds.
		\begin{itemize}
			\item[1.] If $K < 0,$ then $a \in (-\infty, 0),$ $c \in (0, +\infty),$ $f$ is asymptotic to $-\frac{c}{a}$ and
			\begin{itemize}
				\item[$-$] either $I \subseteq \R$ and $f(I) \subseteq (-\frac{c}{a}, +\infty)$ for $-1 \leq a < 0$;
				
				\item[$-$] or $I \subseteq (s_0, +\infty)$ and $f(I) \subseteq (-\frac{c}{a}, -\tfrac{c}{1 + a})$ for $a < -1$.
			\end{itemize}
			
			\item[2.] If $K > 0,$ then $a \in (-1, 0) \cup (0, +\infty),$ $c \in (-\infty, 0)$ and
			\begin{itemize}
				\item[$-$] either $I \subseteq (-\infty, s_0)$ and $f(I) \subseteq (-\frac{c}{1+a}, -\tfrac{c}{a})$ for $a > 0$;
				
				\item[$-$] or $I \subseteq (-\infty, s_0)$ and $f(I) \subseteq (\frac{c}{1-a}, +\infty)$ for $-1<a<0$.
			\end{itemize}
		\end{itemize}
	\end{theorem}
	
	We remark some geometric properties of the complete surfaces that appear in the previous results. One of them says that these surfaces are asymptotic to a right circular cylinder, which is also a rotational Ricci surface. We will call them \textit{funnel-Ricci surfaces}. We only state and argue for the case $af(s) + c > 0$ because the other case is analogous and easily deduced.
	
	\begin{theorem}\label{b=0prop}
		Let $\Sigma_f$ be a rotational Ricci surface in $\R^3$ with negative Gaussian curvature, where $f:\R \to (-\frac{c}{a}, +\infty)$ is a positive smooth function defined implicitly by
		\[af(s) - c\log\left(af(s) + c\right) = a^2s + d\]
		for $0< a \leq 1$ and $c < 0.$ Then:
		\begin{itemize}
			\item[1.] $\Sigma_f$ is a proper surface;
			\item[2.] $\Sigma_f$ is asymptotic to the cylinder $\Ccal_0(-\frac{c}{a})$ when $s\to-\infty.$
		\end{itemize}
	\end{theorem}
	\begin{proof}
		We are going to show that $g$ is unbounded from above and from below. Let $s:(-\frac{c}{a},+\infty) \to \R$ be the inverse function of $f$, which is given as
		\[s(t) = \frac{1}{a^2}\left(at - c\log\left(at + c\right) - d\right)\]
		by equation \eqref{solution}. On one hand, since $f(s(t))=t,$ for $\hat{g} = g\circ s$ we get
		\begin{align*}
			\hat{g}(t) = \int_{t_0}^{t}\sqrt{\frac{\tau^2}{\left(a\tau + c\right)^2}-1}\dif \tau< \int_{t_0}^{t}\frac{\tau}{a\tau + c}\dif \tau=\frac{1}{a}\left(\tau-\frac{c}{a}\log(a\tau+c)\right)\Bigg|_{t_0}^t
		\end{align*}
		for some $t_0 > -\frac{c}{a}.$ Therefore, $\hat{g}(t) \to -\infty$ when $t\to-\frac{c}{a}.$ On the other hand, since $0<at+c < at$ by \eqref{inequalities}, then
		\begin{align*}
			\frac{\sqrt{1-a^2}}{a}(t-t_0) &= \int_{t_0}^{t}\frac{\sqrt{1-a^2}}{a}\dif \tau<  \int_{t_0}^{t}\sqrt{\frac{\tau^2}{\left(a\tau + c\right)^2}-1}\dif \tau=\hat{g}(t)
		\end{align*}
		and consequently $\hat{g}(t)\to +\infty$ when $t\to+\infty.$
		
		The second assertion follows from the fact that $\hat{g}(t) \to -\infty$ as $t\to-\frac{c}{a}$ and $f$ is asymptotic to $-\frac{c}{a}$ as shown above.
	\end{proof}
	
	\begin{figure}[!h]
		\centering
		{\includegraphics[scale=.75]{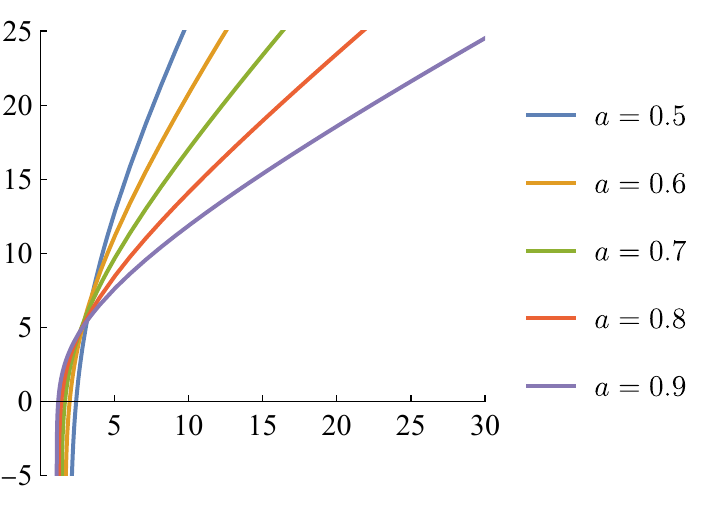}}
		{\includegraphics[scale=.75]{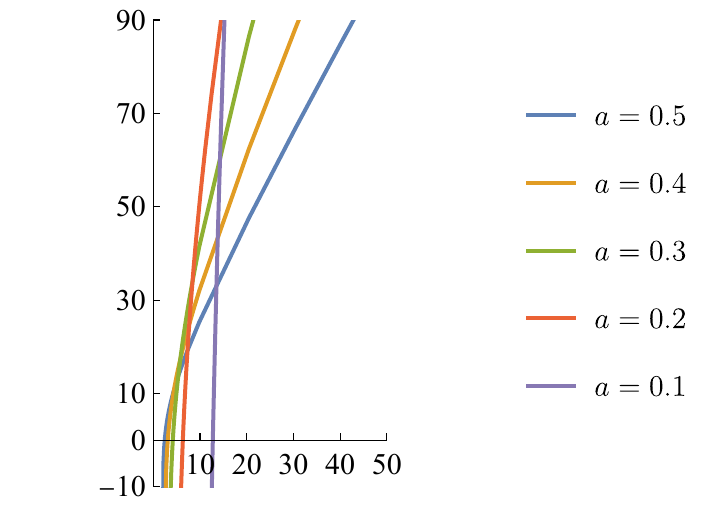}}
		\hfil
		{\includegraphics[scale=.75]{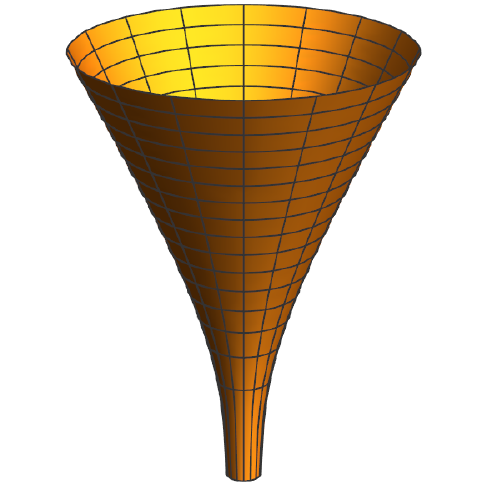}}
		\caption{Funnel-Ricci surface and profile curves for $c=-1$.}
	\end{figure}
	
	\begin{figure}[!h]
		\centering
		{\includegraphics[scale=.75]{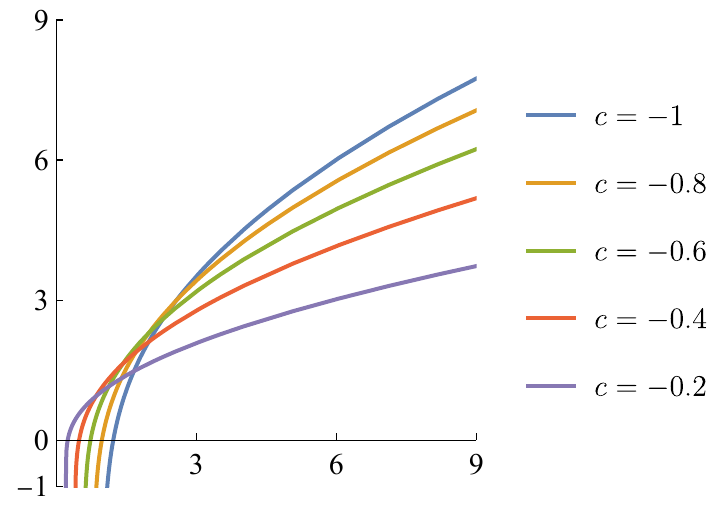}}
		{\includegraphics[scale=.75]{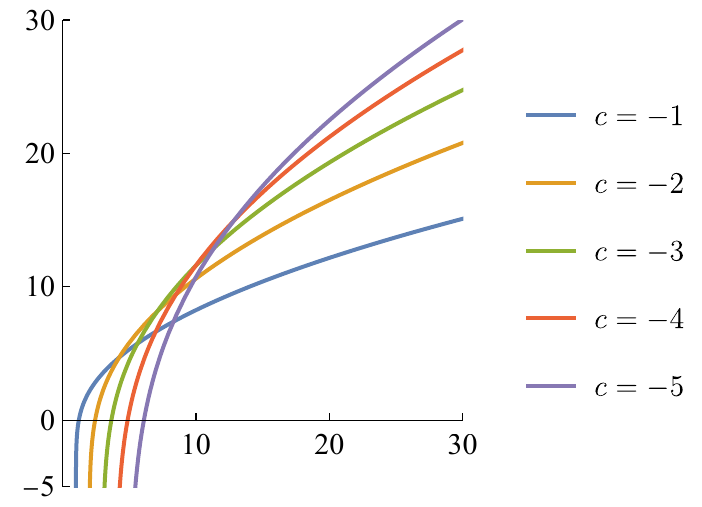}}
		\hfil
		{\includegraphics[scale=.75]{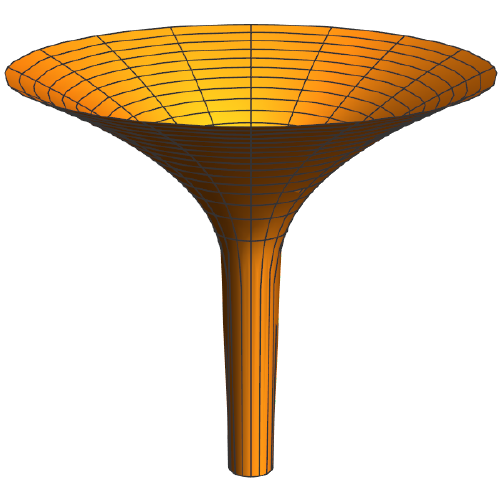}}
		\caption{Funnel-Ricci surface and profile curves for $a=1$.}
	\end{figure}
	
	\subsection{Case \texorpdfstring{$a \neq 0$}{a/=0}  and \texorpdfstring{$b \neq 0$}{b/=0}}
	
	The last possibility that we will consider for the constants of \eqref{ode-ricci} is non-zero $a$ and $b.$ In this case, we have that
	\begin{align*}
		f(s(t)) &= adt\exp\left(-\int_{t_0}^t \frac{\tau}{\tau^2 - \tau - \frac{b}{a^2}}\dif \tau\right),
		\intertext{with}
		s(t) &= d\exp\left(-\int_{t_0}^t \frac{\tau}{\tau^2 - \tau - \frac{b}{a^2}}\dif \tau\right) -\frac{c}{b}.
	\end{align*}
	Moreover, a direct computation using \eqref{eq-K_f} shows that the Gaussian curvature $K$ of $\Sigma_f$ is given by
	\begin{equation}\label{curvature3}
		K(t) = -\frac{b}{a^2d^2t^4}\left(t^2 - t - \frac{b}{a^2}\right) \exp\left(2\int_{t_0}^t \frac{\tau}{\tau^2 - \tau - \frac{b}{a^2}}\dif \tau\right).
	\end{equation}
	
	\begin{theorem}[Case $a \neq 0$ and $b\neq 0$]
		Let $\Sigma_f$ be a rotational Ricci surface in $\R^3,$ where $f:I \to \R$ is a positive smooth function defined by
		\begin{align*}
			f(s(t)) &= adt\exp\left(-\int_{t_0}^t \frac{\tau}{\tau^2 - \tau - \frac{b}{a^2}}\dif \tau\right),
			\intertext{where}
			s(t) &= d\exp\left(-\int_{t_0}^t \frac{\tau}{\tau^2 - \tau - \frac{b}{a^2}}\dif \tau\right) -\frac{c}{b}
		\end{align*}
		for some $a, b, c, d \in \R$ with $a \neq 0$, $b\neq 0$ and  $t_0 \in J := \{t \in \R;\ s(t) \in I\}.$ Suppose that the Gaussian curvature $K$ of $\Sigma_f$ does not vanish identically. Then,
		\[J \subseteq \left\{t\in\R : -1 \leq \frac{b}{at} + a \leq 1\right\} \cap \left\{t\in\R : t^2 - t - \frac{b}{a^2} \neq 0\right\}.\]
	\end{theorem}
	
	\begin{proof}
		Writing $f(t) = f(s(t)),$ Remark \ref{remark-derivadas} implies
		\begin{align*}
			0 &\leq f(t)^2 - \left(af(t) + bs(t) + c\right)^2\\
			&= \left((1 - a)f(t) - bs(t) - c\right)\left((1 + a)f(t) + bs(t) + c\right).
		\end{align*}
		Consequently, in the maximal set of definition of $\alpha,$ we must have that either
		\begin{equation}\label{wrongsystem}
			\begin{cases}
				(1 - a)f(t) - bs(t) - c \leq 0\\
				(1 + a)f(t) + bs(t) + c \leq 0
			\end{cases}
		\end{equation}
		or
		\begin{equation}\label{possiblesystem}
			\begin{cases}
				(1 - a)f(t) - bs(t) - c \geq 0\\
				(1 + a)f(t) + bs(t) + c \geq 0.
			\end{cases}
		\end{equation}
		The system \eqref{wrongsystem} implies $f(t) \leq 0,$ which cannot occur by the assumption that $f$ is positive. Therefore, the system \eqref{possiblesystem} is the only possibility. Replacing the expressions of $f$ and $s$ in its first inequality, we have $bd \leq a(1-a)dt.$ An analogous computation using the second inequality of \eqref{possiblesystem} gives that $-a(1+a)dt \leq bd.$ Therefore, since $0 < f(t) = adt,$ the interval $J$ satisfies
		
		\begin{equation*}\label{inequalities-2}
			J \subseteq \left\{t\in\R : -1 \leq \frac{b}{at} + a \leq 1\right\}.
		\end{equation*}
		Furthermore, setting $R(t) = t^2 - t - \frac{b}{a^2},$ we have
		\[K(t) = -\frac{bR(t)}{a^2d^2t^4}\exp\left(2\int_{t_0}^t \frac{\tau}{\tau^2 - \tau - \frac{b}{a^2}}\dif \tau\right)\]
		by equation \eqref{curvature3}. The discriminant of the polynomial $R$ is $\Delta_R = \frac{a^2 + 4b}{a^2}.$ When $a^2 + 4b < 0,$ there is nothing to do. When $a^2 + 4b \geq 0,$ the roots of $R$ are given by
		\[t_1 = \frac{|a| - \sqrt{a^2 + 4b}}{2|a|}\ \ \text{and}\ \ t_2 = \frac{|a| + \sqrt{a^2 + 4b}}{2|a|}.\]
		By the assumption that $a \neq 0 \neq b,$ we conclude that $t_1 \neq 0 \neq t_2.$ If there is $t_* \in J$ such that $K(t_*) = 0,$ we must have $t_* \in \{t_1, t_2\}$, and consequently that $K$ vanishes identically by Proposition \ref{vetor-horizontal}. Therefore, for each value of $\Delta_R,$
		\[J \subseteq \{t\in\R : R(t) \neq 0\}.\]
		This concludes our assertion.
	\end{proof}
	
	\section{Rotational Ricci surfaces meeting $\partial B^3$ orthogonally}\label{freeboundary}
	
	Let $B^3$ be the unit Euclidean three-dimensional ball centered at the origin of $\R^3,$ and let $\Sigma$ be a compact surface with non-empty smooth boundary $\partial \Sigma$ such that $\intt(\Sigma) \subset \intt(B^3),$ $\partial\Sigma \subset \partial B^3$ and the unit conormal exterior to $\Sigma$ along $\partial \Sigma$ coincides with the position vector. If $\Sigma$ is minimal, it is the so-called \textit{free boundary minimal surface in $B^3.$} The condition along $\partial \Sigma$ will be referred as \textit{orthogonality condition.}
	
	When $\Sigma = \Sigma_f$ is a rotational surface in $B^3,$ one can express the orthogonality condition as follows. An arbitrary boundary curve $\Gamma$ of $\partial \Sigma_f$ can be parametrised by
	\begin{equation}\label{gamma}
		\gamma(\theta) = X\left(\bar{s}, \frac{\theta}{f(\bar{s})}\right),
	\end{equation}
	for some $\bar{s} \in \partial I$. This means that the unit vector $\gamma'(\theta)$ tangent to $\Gamma$ is parallel to $X_{\theta}(\bar{s}, \theta).$ Since $X$ is an orthogonal parametrisation of $\Sigma_f$, we conclude that $X_s(\bar{s}, \theta)$ is normal to $\Gamma.$ Moreover, once the profile curve of $\Sigma_f$ is parametrised by its arc length, then $X_s(\bar{s}, \theta)$ is the unit conormal exterior to $\Sigma_f$ along $\Gamma.$ Hence, looking to $X(\bar{s}, \theta)$ as the position vector, the orthogonality condition along $\Gamma$ is given by
	\[X_s(\bar{s}, \theta) = X(\bar{s}, \theta).\]
	
	It is well-known that there is only a piece of catenoid, up to isometries of $B^3,$ that is free boundary in $B^3.$ It is the well-known \textit{critical catenoid}. The goal of this section is to construct a family of Ricci surfaces with the orthogonality condition that interpolates the geodesic of $B^3$ contained in the $z$-axis and the critical catenoid. Before, we will see that the orthogonality condition imposes strong restrictions on the possibilities of these surfaces.
	
	\begin{lemma}\label{fblemma}
		Let $\Sigma_f$ be a rotational Ricci surface inside $B^3$ which meets $\partial B^3$ orthogonally, where $f:[\bar{s}_1, \bar{s}_2] \to \R$ is a positive smooth function. Then, the Gaussian curvature of $\Sigma_f$ is negative and $f$ is given by
		\[f(s) = \sqrt{bs^2 + 2cs + d}\]
		for some constants $b,c,d \in \R$ with $c^2<bd.$
	\end{lemma}
	
	\begin{proof}
		Let $\Sigma_f$ be a rotational Ricci surface in $B^3$ meeting $\partial B^3$ orthogonally, where $f: [\bs_1, \bs_2] \to \R$ is a positive smooth function. Denote by $\partial \Sigma_f = \Gamma_1 \cup \Gamma_2$ the boundary of $\Sigma_f,$ where $\Gamma_1$ and $\Gamma_2$ are two disjoint circles. Since the profile curve of $\Sigma_f$ is embedded, this surface is topologically an annulus, and hence Gauss-Bonnet theorem gives that
		\[\int_{\Sigma_f} K\dif A + \int_{\partial\Sigma_f} \kappa_g \dif \mu = 0,\]
		where $\kappa_g$ is the geodesic curvature of $\partial\Sigma_f$ as curves in $\Sigma_f$ calculated in the inner direction. By the orthogonality condition, if we parametrise an arbitrary boundary curve of $\Sigma_f$ as in \eqref{gamma}, we get
		\[\kappa_g(\bar{s}) = \langle \nabla_{\gamma'} X_s, \gamma'\rangle = \frac{1}{f(\bar{s})^2} \langle \nabla_{X_\theta} X, X_\theta\rangle = \frac{1}{f(\bar{s})^2} \langle X_\theta, X_\theta \rangle = 1,\]
		where $X_s = X_s(\bar{s}, \theta),$ $X = X(\bar{s}, \theta)$ and $\bar{s} \in \{\bar{s}_1, \bar{s}_2\}.$ Thus,
		\[\int_{\Sigma_f} K\dif A = -L(\partial \Sigma_f),\]
		where $L(\cdot)$ means the length of a curve. This shows that $K < 0.$ In this case, the Ricci condition is rewritten as $4K = \Delta\log(-K).$ Combining divergence theorem and orthogonality condition, we can use the first identity of \eqref{gradient-and-laplacian} and obtain that
		\begin{equation}\label{divergence}
			\begin{aligned}
				-4L(\partial \Sigma_f) &= \int_{\partial \Sigma_f}\langle \nabla\log(-K), X_s\rangle\dif\mu\\
				&= \int_{\partial \Sigma_f} \frac{K'}{K}\dif\mu\\
				&= \frac{K'(\bar{s}_1)}{K(\bar{s}_1)}L(\Gamma_1) + \frac{K'(\bar{s}_2)}{K(\bar{s}_2)}L(\Gamma_2).
			\end{aligned}
		\end{equation}
		Note that the last equality above is valid because $K' = K'(\bar{s})$ and $K = K(\bar{s})$ are constants along any boundary component of $\Sigma_f.$ Assume without loss of generality that
		\[\frac{K'(\bar{s}_1)}{K(\bar{s}_1)} \leq \frac{K'(\bar{s}_2)}{K(\bar{s}_2)}.\]
		At the points of $\Gamma_1,$ equation \eqref{divergence} implies $K'(\bar{s}_1) + 4K(\bar{s}_1) \geq 0.$ By \eqref{eq-2} and orthogonality condition $f'(\bar{s}_1) = f(\bar{s}_1),$ this inequality is equivalent to $f(\bar{s}_1)(f'''(\bar{s}_1) + 3f''(\bar{s}_1)) \leq 0.$ Thus, we can use equation \eqref{prooflemma} and see that
		\[af''(\bar{s}_1) = f(\bar{s}_1)\left(f'''(\bar{s}_1) + 3f''(\bar{s}_1)\right) \leq 0,\]
		where $a$ is the constant of \eqref{ode-ricci}. Since $f'' > 0$ on $I,$ then $a \leq 0.$ Analogously, we can use \eqref{divergence} to see that $K'(\bar{s}_2) + 4K(\bar{s}_2) \leq 0$ and conclude that $a \geq 0,$ which shows that $a = 0.$ It follows from Theorem \ref{a=0} that there are constants $b, c, d \in \R$ with $c^2 < bd$ such that $f(s) = \sqrt{b s^2 + 2c s + d}.$
	\end{proof}
	
	In order to guarantees that the intersection between the rotational Ricci surface and the boundary of $B^3$ is non-empty, we need to restrict ourselves to pieces of catenoidal-Ricci surfaces.
	
	\begin{theorem}\label{fbtheorem}
		There exists an one-parameter family $\{\Sigma^b\}_{0 \leq b \leq 1}$ inside $B^3$ with the following properties:
		\begin{itemize}
			\item[1.] $\Sigma^0$ is the vertical geodesic of $B^3$;
			\item[2.] $\Sigma^b$ is a piece of a catenoidal-Ricci surface meeting $\partial B^3$ orthogonally for $0 < b \leq 1$;
			\item[3.] $\Sigma^1$ is the critical catenoid.
		\end{itemize}
		%Moreover, up to isometries of $B^3,$ the element of the family is unique for any $0 \leq b \leq 1.$
	\end{theorem}
	
	\begin{proof}
		Let $\Sigma_f$ be a rotational Ricci surface in $B^3$ meeting $\partial B^3$ orthogonally. It follows by Lemma \ref{fblemma} that for $b \in (0,1],$ under the change of parameter $s(t) = t -\frac{c}{b},$ the function $f$ becomes the function $\hat{f} = f\circ s: [-\rho, \rho] \to \R$ given by
		\[\hat{f}(t) = \sqrt{b t^2 + \frac{bd - c^2}{b}},\]
		where $\rho>0$ is such that $f(\bs_1) = \hat{f}(-\rho)$ and $f(\bs_2) = \hat{f}(\rho).$ It is direct to check that in the parameter $t,$ equation \eqref{ode-ricci} is rewritten as $\hat{f}(t) \hat{f}'(t) = b t.$ Moreover, since $t$ is the arc length of the profile curve of $\Sigma_f,$ the orthogonality condition implies $\hat{f}'(\rho) = \hat{f}(\rho).$ It follows that
		\begin{equation}\label{r}
			b^2\rho^2 - b^2\rho + bd - c^2 = 0.
		\end{equation}
		This equation gives $0 < b d - c^2 = b^2\rho(1 - \rho)$ and consequently $\rho < 1.$ Thus,
		\[\hat{f}(t) = \sqrt{b}\sqrt{t^2 + \rho(1 - \rho)}.\]
		In this case, we can parametrise $\Sigma_f$ as
		\[X(t, \theta) = \left(\sqrt{b}\sqrt{t^2 + \rho(1 - \rho)} \cos \theta, \sqrt{b}\sqrt{t^2 + \rho(1 - \rho)} \sin \theta, \hat{g}(t)\right),\]
		where $\hat{g} = g\circ s$ is given by
		\[\hat{g}(t) = \int_{0}^t \sqrt{\frac{(1 - b)\tau^2 + \rho(1 - \rho)}{\tau^2 + \rho(1 - \rho)}}\dif \tau.\]
		Note that \eqref{r} has two positive solutions, say $\rho_1 \leq \rho_2.$ Since $\hat{f}$ and $\hat{g}$ are increasing functions for $t > 0$ and $\Sigma_f$ lies in a proper surface by Theorem \ref{a=0prop}, there is only one $\rho \in \{\rho_1, \rho_2\}$ such that
		\begin{equation}\label{boundary}
			1 = \|X(\rho, \theta)\|^2 = b \rho + \hat{g}(\rho)^2
		\end{equation}
		holds. Finally, the orthogonality condition implies $\hat{g}'(\rho) = \hat{g}(\rho),$ and one can check directly that this condition implies equation \eqref{boundary}. This shows existence and uniqueness of $\Sigma_f$ for each $b \in (0,1].$ Moreover, we have that $\hat{f}(t) \to 0$ when $b \to 0.$ Hence, the profile curve of $\Sigma_f$ converges to the straight line contained in the $z$-axis and $\rho \to 1$ as $b \to 0.$
	\end{proof}
	
	\begin{remark}
		It can be checked by means of \eqref{g-function} and \eqref{boundary} that he value $\rho$ for the critical catenoid is the only solution of
		\[\frac{1}{\sqrt{\rho}}\sinh\left(\frac{1}{\sqrt{\rho}}\right) = \frac{1}{\sqrt{1-\rho}},\]
		whose the neck radius is $\sqrt{\rho(1-\rho)}$. Numerically, we have $\rho \approx 0.694817.$
	\end{remark}
	\begin{remark}
		It is well-known that catenoids in $\R^3$ cannot degenerate into a curve. However, Theorem \ref{fbtheorem} produces a family of catenoidal-Ricci surfaces that converges to the $z$-axis of $\R^3.$
	\end{remark}
	
	\bibliographystyle{siam}
	\bibliography{references}
\end{document}